\documentclass[a4paper,12pt]{article}
\usepackage[latin1]{inputenc}
\usepackage[swedish,english]{babel}
\usepackage{amsmath,amsthm}
\usepackage{amssymb}
\usepackage{graphicx}
\usepackage{subfigure}
\usepackage{natbib}

\newcommand{\Ordo}[1]{\mathcal{O} \left( #1 \right)}
\newcommand{\natop}[2]{\genfrac{}{}{0pt}{1}{#1}{#2}}

\newcommand{\Realdom}{\mathbf{R}}
\newcommand{\Intdom}{\mathbf{Z}}
\newcommand{\stoich}{\mathbb{N}}
\newcommand{\hatw}{\hat{w}}

\newcommand{\Probspace}{\Sigma}
\newcommand{\Probelem}{\sigma}
\newcommand{\Probfiltr}{\mathbf{F}}
\newcommand{\Prob}{\mathbf{P}}

\newcommand{\maxW}{\bar{W}}

\newcommand{\Fine}{\mathcal{F}}
\newcommand{\Fineh}{\mathcal{F}^{h}}
\newcommand{\Coarse}{\mathcal{C}}
\newcommand{\err}{e}
\newcommand{\merr}{\bar{e}}
\newcommand{\nsol}{\tilde{X}}
\newcommand{\xsol}{X}
\newcommand{\StabF}{S_{\Fine}}
\newcommand{\StabC}{S_{\Coarse}}
\newcommand{\dt}{\Delta t}

\numberwithin{equation}{section}
\numberwithin{table}{section}
\numberwithin{figure}{section}

\theoremstyle{plain}
\newtheorem{theorem}{Theorem}[section]
\newtheorem{lemma}[theorem]{Lemma}

\newtheorem{proposition}[theorem]{Proposition}

\theoremstyle{definition}
\newtheorem{definition}{Definition}[section]
\newtheorem{assumption}[definition]{Assumption}

\theoremstyle{remark}
\newtheorem*{remark}{Remark}

\title{Parallel in Time Simulation of Multiscale Stochastic Chemical Kinetics}

\author{Stefan Engblom$^{\mbox{\tiny 1}}$ \thanks{Financial support
    has been obtained from the Swedish National Graduate School in
    Mathematics and Computing and the Swedish Foundation for Strategic
    Research.}}

\date{August 27, 2008}

\begin{document}

\selectlanguage{english}

\maketitle

\vspace{-1cm}

\begin{center}
  \footnotesize{\em $^{\mbox{\tiny\rm 1}}$Div of Scientific Computing,
    Dept of Information Technology \\
    Uppsala University, P.~O.~Box 337, SE-75105 Uppsala, Sweden \\
    email: {\tt stefane@it.uu.se} \\ [5pt]}
\end{center}

\begin{abstract}

  A version of the time-parallel algorithm parareal is analyzed and
  applied to stochastic models in chemical kinetics. A fast predictor
  at the macroscopic scale (evaluated in serial) is available in the
  form of the usual reaction rate equations. A stochastic simulation
  algorithm is used to obtain an exact realization of the process at
  the mesoscopic scale (in parallel).

  The underlying stochastic description is a jump process driven by
  the Poisson measure. A convergence result in this arguably difficult
  setting is established suggesting that a homogenization of the
  solution is advantageous. We devise a simple but highly general such
  technique.

  Three numerical experiments on models representative to the field of
  computational systems biology illustrate the method. For non-stiff
  problems, it is shown that the method is able to quickly converge
  even when stochastic effects are present. For stiff problems we are
  instead able to obtain fast convergence to a homogenized solution.

  Overall, the method builds an attractive bridge between on the one
  hand, macroscopic deterministic scales and, on the other hand,
  mesoscopic stochastic ones. This construction is clearly possible to
  apply also to stochastic models within other fields.

\vspace{0.5cm}


\vspace{0.5cm}


%
%
%

\end{abstract}


\section{Introduction}

It has been shown in several studies \cite{stochGene, noisyBusiness,
noisygeneregul, noisygene} that stochastic descriptions of biochemical
networks are necessary for understanding and explaining the mechanisms
inside living cells. Such networks often contain species present in
copy numbers down to a few hundreds \cite{guptasarama}, hinting why
discreteness and stochasticity becomes important. Examples of when
random effects are pronounced include the regularity and drift of
circadian oscillations \cite{circadian2D}, spontaneous separation in
bistable systems \cite{bistable_RDME} and the creation of new
steady-states \cite{newsteadystates_RDME}.

An accurate such stochastic model can be obtained directly from
microphysical considerations as a consequence of the Markov property
\cite{gillespieCME}. Essentially, in a diffusion-controlled system,
reactive collisions between molecules are rare events, implying a
rapid loss of dependence on the systems history.

Sample realizations of any given chemical network can be obtained by
randomly firing reactions and correspondingly updating the systems
state; the most well-known such algorithm is Gillespie's
\emph{stochastic simulation algorithm} (SSA)
\cite{gillespie}. Although conceptually simple, detailed simulation of
complex reaction networks as found for example in living cells remains
a computationally very intensive problem.

The \emph{parareal} algorithm for the solution of evolution equations
on a time-interval $[0,T]$ was suggested in the note \cite{parareal1}
as a method for the \emph{para}llel solution of problems in
\emph{real} time. An improved version, better tuned to nonlinear
problems, was subsequently applied to problems in control theory
\cite{pararealctrl} and molecular dynamics \cite{pararealMol} to
mention just a few.

The method is built around a predictor-corrector step in which a
\emph{coarse} solver is used as a preconditioner to a \emph{fine}
solver. It thus incorporates both multigrid and domain decomposition
ideas, but is unique in that it is a purely parallel algorithm; it has
no value when executed on a serial machine.

The idea suggested in this paper is to use the macroscopic
approximation of the chemical system, the \emph{reaction rate
equations}, as the coarse solver while a stochastic simulation is used
as the fine scale solver. This setup is available also for other types
of stochastic evolution equations but has not received much
attention. The reason is probably that any deterministic model is a
quite poor formal approximation. On the other hand, any macroscopic
trend in the solution is likely to be quickly obtained. Also, with a
deterministic coarse solver, highly efficient implicit integrators are
available which greatly simplifies the implementation.

The present paper is related to references \cite{pararealSDE,
pararealMreduct}. In the former, the analysis of parareal when applied
to ordinary differential equations (ODEs) and stochastic ODEs (SDEs)
driven by Wiener processes is treated. In the latter, the parareal
algorithm is applied to linear mass-balance (deterministic) chemical
kinetics with disparate rates by using a coarse solver in the form of
a reduced model.

The paper is divided into sections as follows: in
Section~\ref{sec:stochkinetics} the mathematical model of stochastic
kinetics is described in some detail and we briefly review the
parareal algorithm in Section~\ref{sec:parareal}. We present a
convergence analysis in Section~\ref{sec:anal}; the main result
indicates that strong convergence might be slow for systems with large
Lipschitz constants. A way around this is suggested in the form of a
simple but very general homogenization procedure in parallel. This
``parareal homogenization'' could well be an interesting approach to
many other stiff problems. We perform three quite different numerical
experiments in Section~\ref{sec:ex} where the practical implementation
of the algorithm is also discussed. Finally, in
Section~\ref{sec:concl} we conclude the paper by discussing the
outcome of the experiments and suggesting some future work and
generalizations.


\section{Stochastic chemical kinetics}
\label{sec:stochkinetics}

In this section we discuss the usual mesoscopic stochastic model for
chemical reactions (the master equation). As it is needed in the
analysis in Section~\ref{sec:anal}, we shall also give an equivalent
and mathematically more appealing representation in terms of a certain
stochastic jump process. For monographs on applications of the master
equation, see \cite{Gardiner, VanKampen}. The mathematical treatment
of jump processes can be found in \cite{LevySDEs, pointProc,
Markovappr, SDEs, SDEsDiffusion}.

\subsection{Chemical reactions, the master equation and a hierarchy of
  methods}

We consider in this paper a general homogenous chemical network
consisting of $D$ different species reacting according to $R$
prescribed reactions. At any given time $t$, the \emph{state} of the
system is an integer vector $x \in \Intdom_{+}^{D} = \{0,1,2, \ldots
\}^{D}$ counting the number of molecules of each kind. In a stochastic
description, each reaction is a change of the state according to a
certain transition rule and the intensity of this process is governed
by a \emph{reaction propensity}, $w_{r}: \Intdom_{+}^{D} \to
\Realdom_{+}$. This is the transition probability per unit of time for
moving from the state $x$ to $x-\stoich_{r}$;
\begin{align}
  \label{eq:prop}
  x &\xrightarrow{w_{r}(x)} x-\stoich_{r},
\end{align}
where by convention, $\stoich_{r} \in \Intdom^{D}$ is the transition
step and is the $r$th column in the \emph{stoichiometric matrix}
$\stoich \in \Intdom^{D \times R}$.

For instance, in a given well-stirred volume $\Omega$ the
\emph{elementary} reactions can be written using the states $x =
[a,b]^{T}$,
\begin{align}
  \label{eq:el1}
  \emptyset \xrightarrow{k_{1} \Omega} A,                               \\
  A \xrightarrow{k_{2} a} \emptyset,                                    \\
  A+A \xrightarrow{k_{3} a(a-1)/\Omega} \emptyset,                      \\
  \label{eq:el4}
  A+B \xrightarrow{k_{4} ab/\Omega} \emptyset.
\end{align}
The propensities are generally scaled such that
\begin{align}
  \label{eq:densitydependent}
  w_{r}(x) &= \Omega u_{r}(x/\Omega)
\end{align}
for some function $u_{r}$ which does not depend on
$\Omega$. Intensities of this form are called \emph{density dependent}
and arise naturally in a number of situations
\cite[Ch.~11]{Markovappr}.

Arguably the most common stochastic description of chemically reacting
systems is the \emph{chemical master equation} (CME). To state it, let
$p(x,t)$ be the probability that a certain number $x$ of molecules is
present at time $t$. The CME \cite{Gardiner, VanKampen} is then given
by
\begin{align}
  \label{eq:Master}
  \frac{\partial p(x,t)}{\partial t} &= 
  \sum_{\natop{r = 1}{x+\stoich_{r}^{-} \ge 0}}^{R}
  w_{r}(x+\stoich_{r})p(x+\stoich_{r},t)-
  \sum_{\natop{r = 1}{x-\stoich_{r}^{+} \ge 0}}^{R}
  w_{r}(x)p(x,t),
\end{align}
where the transition steps are decomposed into positive and negative
parts as $\stoich_{r} = \stoich_{r}^{+}+\stoich_{r}^{-}$.

In a pioneering paper from 1976, Gillespie \cite{gillespie} showed
that exact samples from the master equation can be obtained via the
SSA. This follows essentially by identifying it with the forward
Kolmogorov equation associated to a certain continuous-time Markov
chain to be introduced shortly. Importantly, Gillespie
\cite{gillespieCME} later also showed that the CME is an exact
physical description when the system is well-stirred and in thermal
equilibrium.


For chemical networks where the number of reactions per unit of time
is large, detailed simulation with SSA becomes inefficient. As a
remedy, Gillespie \cite{tau_leap} proposed the \emph{tau-leap method}
with the ability to simultaneously ``leap'' over many reactions at
once. The method has since been developed and improved by several
authors \cite{postleap, tau_li, tau_leap_anal}.


A related issue is \emph{stiffness}: many interesting models become
prohibitively expensive to solve by explicitly simulating the various
involved scales. Several different model reduction techniques have
been proposed for this situation \cite{smalltimesteps, nestedSSA,
haseltine_HSSA}. Also, an implicit version of the tau-leap method has
been developed \cite{impl_tau_leap}, but this method converges in a
very weak sense only \cite{stab_tau_leap, stiffSDEimpl,
tau_leap_anal}.

An alternative approximation is the \emph{Langevin equation}
\cite[Ch.~7]{Markovappr} which is a \emph{continuous} SDE driven by
Wiener processes. On ignoring the noise term and only keeping the
drift, the \emph{reaction rate} equations are obtained. This is a set
of $D$ ODEs approximately evolving the expectation value of the
mesoscopic model. The hierarchy of solution methods (SSA $\to$
tau-leap $\to$ Langevin SDE $\to$ ODE) can thus be thought of as a
transition between the mesoscopic model and the macroscopic one
\cite{tau_leap}.

\subsection{Sample path representation}
\label{subsec:samplepath}

A representation equivalent to, but perhaps more direct than the
master equation~\eqref{eq:Master}, was proposed fairly recently in
\cite{ME2SDE}. The idea is to construct a sample path representation in
the form of jump SDEs driven by Poisson processes. As we shall see and
as pointed out in \cite{tau_li, ME2SDE}, this representation is
particular useful for numerical analysis.

We thus introduce the stochastic variable $X(t) \in \Intdom_{+}^{D}$
counting at time $t \ge 0$ the number of molecules of each
type. Reactions are understood to occur instantly and thus the process
is right continuous only; the notation $X(t-)$ is used to denote the
value of the process prior to any reactions occurring at time $t$. We
assume the existence of a probability space\footnote{Note that,
throughout the paper, the letter $\Omega$ is \emph{not} used for the
probability sample space, but consistently denotes the system volume.}
$(\Probspace,\Probfiltr,\Prob)$ with the filtration $\Probfiltr_{t \ge
0}$ containing $R$-dimensional standard Poisson processes. The
transition probabilities in~\eqref{eq:prop} define
\emph{counting processes} $\pi_{r}(t)$ \cite[Ch.~2.5]{pointProc}
according to
\begin{align}
  E &\left[\pi_{r}(t+dt)-\pi_{r}(t) | \Probfiltr_{t} \right] =
  w_{r}(X(t-))\,dt+o(dt).                                               \\
\intertext{In turn, the counting processes determine the process $X(t)$ 
\cite[Ch.~6.4]{Markovappr},}
  \label{eq:Poissrepr}
  X_{t} &= X_{0}-\sum_{r = 1}^{R} \stoich_{r} \pi_{r}(t).               \\
\intertext{A representation of $\pi_{r}(t)$ in terms of a unit-rate 
Poisson process $\Pi_{r}$ in an \emph{operational} or \emph{scaled} time 
can also be given \cite[Ch.~3.1, 4.4]{pointProc},}
  \label{eq:count}
  \pi_{r}(t) &= \Pi_{r} \left( \int_{0}^{t} w_{r}(X(s-)) \, ds \right).
\end{align}

We now wish to move beyond~\eqref{eq:Poissrepr} and obtain a more
transparent representation in the form of a jump SDE driven by the
proper measure. Processes with nonlinearly dependent jump intensities
are quite difficult to handle since the noise enters with a
complicated dependence on the state. This is in contrast to the more
familiar It\^{o} SDE driven linearly by Wiener processes.

For the present purposes, the Poisson random measure
\cite[Ch.~II.1c]{stochLimits} $\mu(dt \times dz; \, \Probelem)$ with
$\Probelem \in \Probspace$ defines an increasing sequence of arrival
times $\tau_{i} \in \Realdom_{+}$ with corresponding ``marks'' $z_{i}$
uniformly distributed in $[0,\maxW]$ \cite[Ch.~1.4]{pointProc}, where
$\maxW$ is to be defined shortly. The deterministic intensity of
$\mu(dt \times dz)$ is the Lebesgue measure, $m(dt \times dz) = dt
\times dz$ \cite[Part~II, Ch.~2§4]{SDEs}.

The propensities have been general nonlinear functions up to this
moment but we need to impose conditions so as to bound the total
intensity of all reaction channels. It is not difficult to see that a
\emph{finite} intensity implies a bounded solution in the mean square 
sense for finite times (see Lemma~\ref{lem:localstd} below). For
convenience, we take the approach in \cite{tau_li} and formally
specializes the investigation to closed systems. This implies the
existence of a bounded set $S$ containing at any time $t$ the state of
the system and is reasonable from physical considerations; ---any real
application must necessarily involve a bounded total number of
molecules.

Summarizing, we thus have
\begin{assumption}
  \label{ass:bounded}
  The state of the chemical network satisfies $X(t) \in S \subset
  \Intdom_{+}^{D}$. When confined to this set, all propensities are
  Lipschitz continuous in their argument with respect to the Euclidean
  norm, $|w_{r}(x)-w_{r}(y)| \le L_{r}\|x-y\|$. It follows that all
  propensities are bounded such that
\begin{align}
  \maxW_{r} := \sum_{s = 1}^{r} \max_{x \in S} w_{s}(x)
\end{align}
  exists and is bounded for all $r$. We define also the numbers $L =
  \sum_{r} L_{r}$, $\maxW = \maxW_{R}$ and $W = \max_{x \in S}
  \sum_{r} w_{r}(x)$. Note that $W \le \maxW$.
\end{assumption}

The frequency of each reaction is controlled through a set of
indicator functions $\hatw_{r} : S \times [0,\maxW] \to \{0,1\}$
defined as follows:
\begin{align}
  \hatw_{r}(x; \, z) &= \left\{ \begin{array}{l}
    1 \quad \mbox{ if } 0 \le z-\maxW_{r-1} < w_{r}(x),                 \\
    0 \quad \mbox{ otherwise.}
  \end{array} \right.
\end{align}
We can now represent the counting process~\eqref{eq:count} in terms of
the Poisson random measure \cite[Ch.~II.1d]{stochLimits} via a
\emph{thinning} of the measure using an acceptance-rejection strategy
\cite[Ch.~4.3]{pointProc}:
\begin{align}
  \label{eq:countrepr}
  \pi_{r}(t) &= \int_{0}^{t} \int_{0}^{\maxW}
  \hatw_{r}(X(t-); \, z) \, \mu(dt \times dz).
\end{align}
Note that this thinning is not the same as in \cite{tau_li, ME2SDE}
where instead $z \in [0,W]$; the present construction yields sharper
estimates when comparing trajectories formed from different initial
data.

The representation~\eqref{eq:countrepr} combined
with~\eqref{eq:Poissrepr} leads us to the sample path representation
in terms of a Skorohod jump SDE \cite[Ch.~IV.9]{SDEsDiffusion},
\begin{align}
  \label{eq:SDE}
  dX_{t} &= -\sum_{r = 1}^{R} \stoich_{r} \int_{0}^{\maxW}
  \hatw_{r}(X(t-); \, z) \, \mu(dt \times dz).
\end{align}
Often one is interested in separating~\eqref{eq:SDE} into its
``drift'' and ``jump'' terms \cite[Part~II, Ch.~2§4]{SDEs},
\begin{align}
  \nonumber
  dX_{t} &= -\sum_{r = 1}^{R} \stoich_{r} w_{r}(X(t-)) \, dt-
  \sum_{r = 1}^{R} \stoich_{r} \int_{0}^{\maxW} \hatw_{r}(X(t-); \, z)
  (\mu-m)(dt \times dz)                                                 \\
  \label{eq:SDEsplit}
  &=: dX_{D}(t)+dX_{J}(t). 
\end{align}
We will use the initial value convention $X_{D}(0) = X(0)$ and
$X_{J}(0) = 0$ throughout the paper. Note that the second term
in~\eqref{eq:SDEsplit} driven by the compensated measure $(\mu-m)$ is
a martingale of mean zero. On taking expectation values and
approximating, we obtain $dEX_{t} \approx d\tilde{x}(t)$ where
\begin{align}
  \label{eq:rate}
  d\tilde{x}(t) &= -\sum_{r = 1}^{R} \stoich_{r} w_{r}(\tilde{x}(t)) \, dt
\end{align}
and $\tilde{x}(0) = EX(0)$. This approximation constitutes, up to a
scaling by the system volume $\Omega$, the usual macroscopic
reaction-rate equations.


\section{The parareal algorithm}
\label{sec:parareal}

We continue by giving a short account of the parareal algorithm. For a
quick introduction with additional references the survey
\cite{parareal} can be recommended. See also \cite{stabPararealPDE,
pararealSurv, pararealAnal} for abstract convergence results,
stability analysis and more.

Consider the general time-dependent problem written in operator form,
\begin{align}
  \label{eq:gp}
  \dot{u} &= -Au, t \in [0,T] \mbox{ with } u(0) = u_{0}.
\end{align}
Write the solution propagator $\Fine$ as
\begin{align}
  \Fine_{t}(y) &= y-\int_{0}^{t}Au(t) \, dt,
\end{align}
where $u$ solves~\eqref{eq:gp} subject to $u(0) = y$. The parareal
algorithm assumes that a coarse solver $\Coarse$ is available that
does the same thing but faster and presumably less accurate.

Time is now discretized in $N = T/\dt$ chunks of time and any solver
$S \in \{\Fine_{\dt},\Coarse_{\dt}\}$ can be used to compute a
numerical solution:
\begin{align}
  \left[ \begin{array}{cccc}
      I  & 0  &  0 &  0                                                 \\
      -S & I  &  0 &  0                                                 \\
      0  & -S &  I &  0                                                 \\
      0  & 0  & -S &  I
    \end{array} \right] \left[ \begin{array}{c}
      v_{0}                                                             \\
      v_{1}                                                             \\
      v_{2}                                                             \\
      v_{3}
    \end{array} \right]
  &= \left[ \begin{array}{c}
      u_{0}                                                             \\
      0                                                                 \\
      0                                                                 \\
      0
    \end{array} \right],
\end{align}
or simply $B(S)v = u_{0}$ in matrix notation. The parareal algorithm
emerges as the fix-point iteration obtained by using
$B(\Coarse_{\dt})^{-1}$ as an approximate inverse to $B(\Fine_{\dt})$:
\begin{align}
  v_{k+1} &= v_{k}-B(\Coarse_{\dt})^{-1}(B(\Fine_{\dt})v_{k}-u_{0}).
\end{align}
Let $v_{0,0} = u_{0}$ and $v_{0,n} = \Coarse_{\dt}v_{0,n-1}$ for
$n \in [1,N]$ to start up the algorithm. One readily verifies the
recursion
\begin{align}
  v_{k,n} = \Fine_{\dt} v_{k-1,n-1}-
  [\Coarse_{\dt} v_{k-1,n-1}-\Coarse_{\dt} v_{k,n-1}],
\end{align}
where the evaluation of $\Fine$ is trivially parallel.

At any point in the algorithm, the preconditioned residual is given by
\begin{align}
  \label{eq:pres}
  p_{k} &= v_{k}-v_{k+1},
\end{align}
and is useful as an estimate of the error.

A nice survey of the parareal method is found in \cite{pararealSurv}
where the connection to some earlier methods is made. It turns out
that one can view the process as a Newton method applied to a certain
nonlinear problem using a Jacobian approximated by finite
differences. This partially explains the superlinear convergence often
observed on bounded intervals. It is known that the parareal algorithm
converges poorly for problems where the operator has eigenvalues with
large imaginary parts \cite{pararealAnal}. Another interesting result
is that the parareal scheme can be made unconditionally stable by
using a suitable relaxation of the coarse solver
\cite{stabPararealPDE}.


\section{Analysis}
\label{sec:anal}

In this section we analyze the parareal algorithm obtained by using
the rate equations~\eqref{eq:rate} and the jump SDE~\eqref{eq:SDE} as
the coarse and fine solver, respectively. In
Section~\ref{subsec:prels} we give some results that cover local
properties of solutions to these equations. The actual convergence
analysis is found in Section~\ref{subsec:convergence} where
convergence in mean square and convergence of the first moment is
investigated. As noted in \cite{pararealSDE}, weak convergence is not
so interesting for the parareal algorithm since weak estimates
typically come into play when many trajectories are simultaneously
computed in a Monte Carlo simulation; parallelism can then be
trivially achieved with optimal efficiency. Our interest in the
convergence of the first moment mainly stems from the possibility to
homogenize the model on the fly by a suitable averaging filter. This
idea is discussed in Section~\ref{subsec:homogenization}.

Because the process $X(t)$ is integer valued, we mention here for
clarity the trivial inequality
\begin{align}
  \label{eq:intineq}
  |X| \le X^{2} &\Longrightarrow E|X| \le EX^{2}.                       \\
\intertext{When $X$ is large, \eqref{eq:intineq} is an overestimate and 
we will switch to the standard inequality}
  \label{eq:CSineq}
  E|X| &\le (EX^{2})^{1/2}.
\end{align}
For simplicity and without serious loss of generality, we treat the
1-D case only.

\subsection{Preliminaries}
\label{subsec:prels}

We start by citing the following convenient lemma.

\begin{lemma}[Lemma~3.1 and 3.2 in \cite{tau_li}; see also Lemma~2.3.2 in 
  \cite{LevySDEs}]
  \label{lem:jump}
  Define $\Delta X_{t} = X(t)-X(t-)$. Then for any fixed time $s > 0$,
  $\Delta X_{s} = 0$ (a.s.). Furthermore, if $w$ is a continuous
  function and $t_{1} < t_{2}$, then
\begin{align}
  \int_{t_{1}}^{t_{2}} \Delta w(X_{t}) \, dt &= 0.
\end{align}
\end{lemma}

We continue by establishing two lemmas concerning the stability of the
jump process $X(t)$.

\begin{lemma}
  \label{lem:localstd}
  Let $X(t \ge 0)$ evolve according to~\eqref{eq:SDE} and define the
  jump term $X_{J}(t)$ as in~\eqref{eq:SDEsplit} with $X_{J}(0) =
  0$. Then
\begin{align}
  \label{eq:localstd1}
  E X_{J}(t)^{2} &\le \|\stoich\|^{2}W t.                               \\
\intertext{Furthermore,}
  \label{eq:localstd2}
  E [X_{t}-X_{0}]^{2} &\le 2\|\stoich\|^{2}W t+
  2\|\stoich\|^{2}W^{2} t^{2}.
\end{align}
\end{lemma}

\begin{proof}
The left side of~\eqref{eq:localstd1} can be written
\begin{align*}
  &\phantom{\le} E\left( \int_{0}^{t} \int_{0}^{\maxW} \sum_{r=1}^{R}
  -\stoich_{r} \hatw_{r}(X(s-); \, z)(\mu-m)(ds \times dz) \right)^{2}  \\
  &= E \int_{0}^{t} \int_{0}^{\maxW} \left( \sum_{r=1}^{R}
  -\stoich_{r} \hatw_{r}(X(s-); \, z) \right)^{2} m(ds \times dz)       \\
  &\le E \int_{0}^{t} \int_{0}^{\maxW} \|\stoich\|^{2} \sum_{r=1}^{R}
  \hatw_{r}(X(s-); \, z)^{2} \, m(ds \times dz),
\end{align*}
where the It\^{o} isometry for jump processes was used \cite[Part~II,
Ch.~2§5]{SDEs} \cite[Lemma~4.2.2]{LevySDEs}. Using that $\hatw_{r}^{2}
= \hatw_{r}$ we complete the proof
of~\eqref{eq:localstd1}. For~\eqref{eq:localstd2} we use the
drift/jump split~\eqref{eq:SDEsplit} and the inequality $(a+b)^{2} \le
2a^{2}+2b^{2}$ to obtain
\begin{align*}
  E [X_{t}-X_{0}]^{2} &\le 2EX_{J}(t)^{2}+2E\left( \int_{0}^{t} 
  \sum_{r=1}^{R} -\stoich_{r} w_{r}(X(s-)) \, ds \right)^{2}.
\end{align*}
\end{proof}

\begin{lemma}
  \label{lem:Cstab}
  Let the drift term be defined as in~\eqref{eq:SDEsplit} with
  $X_{D}(0) = X(0)$ and let $\tilde{x}(t)$ be the reaction rate
  approximation~\eqref{eq:rate} to the expected value of $X(t)$ with
  $\tilde{x}(0) = X(0)$. Then
\begin{align}
  \label{eq:Cstab}
  E [X_{D}(t)-\tilde{x}(t)]^{2} &\le
  L^{2}\|\stoich\|^{4}W t^{3} \exp (2L^{2}\|\stoich\|^{2}t^{2}).
\end{align}
\end{lemma}

\begin{proof}
By Jensen's inequality and the Lipschitz boundedness we get
\begin{align*}
  &\phantom{\le} E[ X_{D}(t)-\tilde{x}(t) ]^{2} \le
  L^{2}\|\stoich\|^{2} t E \int_{0}^{t} [X(s)-\tilde{x}(s)]^{2} \, ds   \\
  &\le L^{2}\|\stoich\|^{2} t \int_{0}^{t}
  2EX_{J}(s)^{2}+2E[X_{D}(s)-\tilde{x}(s)]^{2} \, ds                    \\
  &\le L^{2}\|\stoich\|^{4} Wt^{3}+
  2L^{2}\|\stoich\|^{2}t \int_{0}^{t}
  E[X_{D}(s)-\tilde{x}(s)]^{2} \, ds,
\end{align*}
where the first part of Lemma~\ref{lem:localstd} was used. The result
now follows as an application of Grönwall's inequality.
\end{proof}

%

\begin{remark}
  In a closed system we can identify the magnitude of $X$ with the
  system volume $\Omega$; $\max_{x \in S} \|x\| = \Ordo{\Omega}$. For
  density dependent propensities (see~\eqref{eq:densitydependent}),
  this implies that $W \sim \Omega$ while $L$ does not scale with
  $\Omega$. Lemma~\ref{lem:localstd} and~\ref{lem:Cstab} then yields
\begin{align}
  \label{eq:rrconvergence}
  \left( E[X(t)-\tilde{x}(t)]^{2} \right)^{1/2} &\le C_{t} \Omega^{1/2}
\end{align}
  for some bounded constant $C_{t}$. By Chebyshev's inequality we thus
  have convergence in probability $\Omega^{-1}X(t) \to
  \Omega^{-1}\tilde{x}(t)$ as $\Omega \to \infty$. For more precise
  reasoning of this type, see \cite{Markovappr}.
\end{remark}

We proceed with a result that is crucial to the convergence properties
of the parareal algorithm.

\begin{theorem}
  \label{th:CMCstab}
  Let $X(t \ge 0)$ and $Y(t \ge 0)$ evolve according to~\eqref{eq:SDE}
  using the same underlying set of Poisson processes but with
  different initial values $X_{0}$ and $Y_{0}$. Define the jump terms
  $X_{J}(t)$ and $Y_{J}(t)$ as in Lemma~\ref{lem:localstd}. Then
\begin{align}
  \label{eq:CMCstab}
  E \left[ Y_{J}(\dt)-X_{J}(\dt) \right]^{2} &\le
  L\|\stoich\|^{2} \dt |Y_{0}-X_{0}|(1+\Ordo{\dt}).
\end{align}
\end{theorem}

\begin{proof}
By the It\^{o} isometry we can bound the left side
of~\eqref{eq:CMCstab} by
\begin{align*}
  &\phantom{\le}E \int_{0}^{\dt} \int_{0}^{\maxW} \|\stoich\|^{2}
  \sum_{r=1}^{R} \left[ \hatw_{r}(Y(t-); \, z)-
  \hatw_{r}(X(t-); \, z) \right]^{2} m(dt \times dz)                    \\
  &\le E \, \|\stoich\|^{2} \int_{0}^{\dt} \sum_{r=1}^{R}
  |w_{r}(Y(t-))-w_{r}(X(t-))| \, dt,
\end{align*}
where the usefulness of the thinning proposed in
Section~\ref{subsec:samplepath} is evident. From the Lipschitz
assumption and Lemma~\ref{lem:jump} we get the bound
\begin{align*}
  &\le L\|\stoich\|^{2} \int_{0}^{\dt}
  E \left| Y(t-)-X(t-) \right| \, dt
  = L\|\stoich\|^{2} \int_{0}^{\dt}
  E \left| Y(t)-X(t) \right| \, dt                                      \\
  &\le L\|\stoich\|^{2} \dt \left| Y_{0}-X_{0} \right|+
  L\|\stoich\|^{2} \int_{0}^{\dt} E \left| Y(t)-Y_{0} \right|+
  E\left| X(t)-X_{0} \right| \, dt                                      \\
  &\le L\|\stoich\|^{2} \dt |Y_{0}-X_{0}|+
  4L \|\stoich\|^{4} \int_{0}^{\dt}  Wt+W^{2}t^{2} \, dt,
\end{align*}
where the integer inequality~\eqref{eq:intineq} and the second part of
Lemma~\ref{lem:localstd} was used in the last line.
\end{proof}

The following result, similar in spirit to Theorem~\ref{th:CMCstab},
is useful for studying the error in the first moment alone.
\begin{theorem}
  \label{th:Ccons}
  Define $Y_{D}(t)$, $\tilde{y}(t)$, $X_{D}(t)$ and $\tilde{x}(t)$ as
  in Lemma~\ref{lem:Cstab}. Then
\begin{align}
  \label{eq:Ccons}
  |EY_{D}(\dt)-\tilde{y}(\dt)-EX_{D}(\dt)+\tilde{x}(\dt)| &\le
  \Ordo{\dt^{3/2}}|E[Y_{0}-X_{0}]|.
\end{align}
\end{theorem}

\begin{proof}
If $Y_{0} = X_{0}$ there is nothing to prove and we may thus assume
that $|Y_{0}-X_{0}| \ge 1$. Two applications of Lemma~\ref{lem:Cstab}
gives that
\begin{align*}
  Y_{D}(\dt)-\tilde{y}(\dt)-X_{D}(\dt)+\tilde{x}(\dt) &= \Ordo{\dt^{3/2}}.
\end{align*}
We recover \eqref{eq:Ccons} by inserting the factor $(Y_{0}-X_{0})$,
taking expectation and absolute values of both sides.
\end{proof}

\subsection{Convergence of the parareal algorithm}
\label{subsec:convergence}

For the application of the parareal algorithm we let $\Fine y$ be the
stochastic evolution in a time-step $\dt$ of the jump process $X(t)$
obeying~\eqref{eq:SDE} with initial data
$y$. Following~\eqref{eq:SDEsplit}, we write $\Fine =
\Fine_{D}+\Fine_{J}$ for the drift- and jump term respectively. The
coarse solver $\Coarse y$ is instead the evolution in a time-step
$\dt$ using the deterministic rate equations~\eqref{eq:rate} starting
from $y$.

We denote by $\xsol = [\xsol_{0},\ldots,\xsol_{N}]$ with $\xsol_{n} =
\xsol(t_{n})$ the exact solution at times $t_{n} \equiv n \dt$.
$\nsol_{k,n} \approx \xsol_{n}$ is the numerical approximative
solution obtained after the $k$th iteration of the parareal algorithm:
\begin{align}
  \nsol_{k,0} &= \xsol_{0}, \quad k \ge 0,                              \\
  \nsol_{0,n} &= \Coarse \nsol_{0,n-1}, \quad n \ge 1,                  \\
  \label{eq:parareal_setup}
  \nsol_{k,n} &= \Fine \nsol_{k-1,n-1}-\Coarse \nsol_{k-1,n-1}+
  \Coarse \nsol_{k,n-1},
\end{align}
where $(k,n) \ge 1$ in~\eqref{eq:parareal_setup}.

We now wish to analyze the root-mean-square (RMS) error defined by
\begin{align}
  \err_{k,n}^{2} &= E [\nsol_{k,n}-\xsol_{n}]^{2}.                      \\
\intertext{This measure satisfies by~\eqref{eq:parareal_setup} the 
  recursion}
  \nonumber
  \err_{k,n}^{2} &= E [T_{1}+T_{2}+T_{3}]^{2}                           \\
  \label{eq:errsplit}
  &= E \left[ T_{1}^{2}+T_{2}^{2}+T_{3}^{2}+
  2T_{1}T_{2}+2T_{1}T_{3}+2T_{2}T_{3} \right],                          \\
\intertext{where in terms of}
  T_{1} &\equiv \Fine_{J} \nsol_{k-1,n-1}-\Fine_{J} \xsol_{n-1},        \\
  T_{2} &\equiv \Fine_{D} \nsol_{k-1,n-1}-\Coarse \nsol_{k-1,n-1}-
  \Fine_{D} \xsol_{n-1}+\Coarse \xsol_{n-1},                            \\
  T_{3} &\equiv \Coarse \nsol_{k,n-1}-\Coarse \xsol_{n-1}.
\end{align}
By Theorem~\ref{th:CMCstab} using~\eqref{eq:intineq} and
Lemma~\ref{lem:Cstab} we get
\begin{align}
  \label{eq:T_1}
  ET_{1}^{2} &\le L\|\stoich\|^{2} \dt
  \err_{k-1,n-1}^{2}(1+\Ordo{\dt}),                                     \\
  ET_{2}^{2} &\le \Ordo{\dt^{3}}.                                       \\
\intertext{We also have the Lipschitz bound}
  ET_{3}^{2} &\le \exp(2L\|\stoich\|\dt) \err_{k,n-1}^{2}.
\end{align}
As for the cross-terms we have that $ET_{1}T_{3} = 0$ since $T_{1}$ is
a martingale of zero mean and $T_{3}$ is a deterministic
(non-anticipating) function. The remaining terms can be easily
estimated as
\begin{align}
  2ET_{1}T_{2} &\le \dt ET_{1}^{2}+\dt^{-1} ET_{2}^{2}
  \le \Ordo{\dt^{2}} \err_{k,n-1}^{2},                                  \\
  2ET_{2}T_{3} &\le \dt^{-1} ET_{2}^{2}+\dt ET_{3}^{2}
  \le \Ordo{\dt^{2}}+\Ordo{\dt} \err_{k,n-1}^{2}.
\end{align}
Summarizing, we have from~\eqref{eq:errsplit} upon ignoring higher
order terms in $\dt$,
\begin{align}
  \nonumber
  \err_{k,n}^{2} &\le L\|\stoich\|^{2} \dt \err_{k-1,n-1}^{2}+
  \exp(2L\|\stoich\|\dt) \err_{k,n-1}^{2}                               \\
  \label{eq:err1}
  &=: \dt\StabF^{2} \err_{k-1,n-1}^{2}+\exp(2\dt\StabC) \err_{k,n-1}^{2}.
\end{align}
Define $M = \max_{n} \err_{0,n}$ in view
of~\eqref{eq:rrconvergence}. We readily recognize the binomial
recurrence in~\eqref{eq:err1} so that
\begin{align}
  \label{eq:errbound}
  \err_{k,n}^{2} &\le M^{2} \binom{n}{k} \dt^{k} \StabF^{2k}
  \exp(2\dt \StabC (n-k)).
\end{align}
If we now look at iteration $k = \Ordo{1}$ and interval $n = N =
T/\dt$,~\eqref{eq:errbound} can be simplified into
\begin{align}
  \label{eq:convergence1}
  \err_{k,n} &\le C_{1,T} \StabF^{k},
\end{align}
where $C_{1,T}$ is a bounded constant for any given total time
$T$. Eq.~\eqref{eq:convergence1} shows mean square convergence for
contractive problems where $\StabF < 1$ only; it is unclear what
happens for systems with larger Lipschitz constants. A key to
understanding how the analysis can be refined lies in the fact that
the integer inequality~\eqref{eq:intineq} had to be applied before
using Theorem~\ref{th:CMCstab} and arriving at~\eqref{eq:T_1}. For an
initial large error, the inequality~\eqref{eq:CSineq} is sharper but
introduces a nonlinear dependence:
\begin{align}
  \label{eq:err2}
  \err_{k,n}^{2} &\le \dt\StabF^{2} \err_{k-1,n-1}+
  \exp(2\dt\StabC) \err_{k,n-1}^{2}.
\end{align}
This motivates our interest in the next proposition.

\begin{proposition}
  \label{prop:nlinrec}
  Consider for $(k,n) \ge 0$ a sequence $a_{k,n}$ of nonnegative
  numbers. Let $a_{k,0} = 0$ for $k \ge 0$ and suppose that $a_{0,n}
  \le C$ for $n > 0$ and some nonnegative constant $C$. Then for $k
  \le n$ the inequality
\begin{align}
  \nonumber
  a_{k,n} &\le A a_{k-1,n-1}^{1/2}+B a_{k,n-1}                          \\
\intertext{where $A \ge 0$ and $B \ge 1$ is satisfied by}
  \label{eq:nlinbound}
  a_{k,n} &\le A^{2^{2-2^{1-k}}}B^{n-k} (n-k)^{2^{2-2^{1-k}}} C^{2^{-k}}.
\end{align}
\end{proposition}
A proof by induction is easily constructed once the functional form
in~\eqref{eq:nlinbound} has been obtained.

To apply Proposition~\ref{prop:nlinrec} we put $A = \dt\StabF^{2}$, $B
= \exp(2\dt\StabC)$ and identify $\err_{k,n} = a_{k,n}^{1/2}$ so that
$C = M^{2}$. The result is for $k \ll n = N$ that
\begin{align}
  \label{eq:convergence2}
  \err_{k,n} &\le C_{2,T} M^{2^{-k}},
\end{align}
for some constant $C_{2,T}$. During the first few iterations, this
bound is sharper than~\eqref{eq:convergence1} which is only valid when
the error is already small. For instance, if as
in~\eqref{eq:rrconvergence} we have that $M \sim \Omega^{1/2}$, then
$\err_{k,N} \sim [\Omega^{1/2},\Omega^{1/4},\Omega^{1/8},\ldots]$ for
the first few iterations $k = [0,1,2,\ldots]$.

As it provides us with additional insight, we also comment on the
\emph{error in the first moment}. Define for this purpose $\merr_{k,n}
= |E[\nsol_{k,n}-\xsol_{n}]|$. Proceeding as before and using
Theorem~\ref{th:Ccons} it is not difficult to see that this error
satisfies
\begin{align}
  \nonumber
  \merr_{k,n} &\le \Ordo{\dt^{3/2}}\merr_{k-1,n-1}+
  \exp(L\|\stoich\|\dt)\merr_{k,n-1}                                    \\
  \label{eq:wconvergence}
  &\lesssim \binom{n}{k} \dt^{3k/2} \le C_{3,T} \dt^{k/2}.
\end{align}

The estimates~\eqref{eq:convergence1}, \eqref{eq:convergence2}
and~\eqref{eq:wconvergence} show that \emph{for $\dt$ sufficiently
small, the contractive parts converge in mean square while the
non-contractive parts converge in the sense of~\eqref{eq:wconvergence}
only}. This interpretation opens up for modifying the scheme in such a
way that the fast scales are actively filtered away.

\subsection{Homogenization}
\label{subsec:homogenization}

The previous analysis suggests that problems with large Lipschitz
constants may be problematic to solve by the proposed parareal
algorithm. On the other hand, for certain problems involving very
rapid scales, pathwise convergence may not be so interesting to
obtain. Rather are we interested in convergence to a
\emph{homogenized} model which by itself is a weak approximation to
the original system. For example, this situation occurs in the current
context when rapid reaction channels almost balance each other so that
the interesting dynamics occurs on a much slower scale.

Unlike the deterministic case where implicit solvers generally evolve
stiff problems efficiently, it has been proposed that model reduction
techniques are necessary in the stochastic setting
\cite{stiffSDEimpl}. For instance, the implicit tau-leap method has
only been shown to be weakly convergent under the rather strong
assumption of linear propensities \cite{stab_tau_leap, tau_leap_anal}.

Provided that the coarse solver is sufficiently dissipative, the
biggest challenge in directly using the current parareal algorithm for
stiff problems is to detect when the numerical solution is of
sufficient quality; since convergence is at best weak, any jumps on
the order of the noise associated with the fast scale could in
principle be accepted.

An easy way around this is to replace the fine propagator $\Fine$ with
a homogenized version $\Fineh$. A simple example is
\begin{align}
  \label{eq:Fineh}
  \Fineh X_{0} &:= \frac{1}{\delta t}
  \int_{\dt-\delta t}^{\dt} Y(t) \, dt, \quad
  \mbox{where } Y(t) = \Fine_{t} X_{0},
\end{align}
i.e.~an average of the exact trajectory over an interval $\delta
t$. This interval should be large enough to contain several fast
reactions but short enough to be essentially independent on the slow
scales. One can easily think of much more advanced filters to create a
suitable homogenization but we settled for the
immediate~\eqref{eq:Fineh}.

Let us remark that an essentially ``exact'' trajectory is available
within the homogenized solution since the non-averaged solution is
always computed first. This ``exact'' trajectory is just as noisy as
the true solution but allows for jumps related to the fast scale in
between successive time intervals $[t_{n},t_{n+1}]$. An intuitive
interpretation is that the homogenized trajectory is an exact sample
from a nearby model containing additional reactions that are scheduled
at deterministic time-steps $\dt$. In principle, the intensity of this
unknown process could be estimated and be put in relation to the rest
of the propensities.


\section{Numerical examples}
\label{sec:ex}

After discussing the implementation of the proposed method, we will
consider three numerical examples. The first is a typical example of
when stochastic models are necessary since mean-field equations
generally give incorrect results. Nevertheless, when they are used as
a preconditioner in the parareal algorithm, convergence to the true
solution is quite fast. The second example is representative of
situations involving multiple scales with fast transients. Finally, we
obtain a trajectory from the \emph{reaction-diffusion} master equation
which is a good representative of very large networks. Here, the
macroscopic model is just the familiar reaction-diffusion partial
differential equation.

\subsection{Implementation}

For the purpose of performing experiments we have implemented a serial
version of the parareal algorithm in Matlab. The coarse solver is thus
simply the reaction rate equations solved by any suitable ODE-solver;
we typically used Matlab's \texttt{ode23s} but occasionally got away
with just a single step of the backward Euler method. As the fine
solver we used an implementation of the \emph{next reaction method}
(NRM) \cite{gillespiemodern}, which is statistically equivalent to the
SSA but specifically tuned to large networks. The reason for
preferring this particular implementation is that \emph{consistent
  Poisson processes may easily be sampled}. This feature is achieved
by simply using the same sequence of random numbers for each reaction
channel.

To appreciate where this technicality enters in the analysis, note
that in Theorem~\ref{th:CMCstab}, $Y_{t}$ and $X_{t}$ are assumed to
be trajectories conditioned on the same Poisson process. If they are
not related in this respect, the associated constants are much larger.

In order to estimate the quality of the numerical solution we used the
relative Euclidean norm of the preconditioned residual as given
by~\eqref{eq:pres};
\begin{align}
  \label{eq:presidual}
  \err_{k+1} &\sim \max_{n} D^{-1/2} \|
  (v_{k,n}-v_{k+1,n})/(1+v_{k+1,n}) \|,                                 \\
\intertext{(elementwise division). This was then taken as an 
  approximation to the true relative error}
  \label{eq:error}
  \err_{k+1} &= \max_{n} D^{-1/2} \|
  (v_{k+1,n}-u_{n})/(1+u_{n}) \|.
\end{align}

We generally did not round any fractional results obtained from the
coarse solver. Formally, this introduces a complication in the
analysis as the solution space becomes a continuum (see
\cite[Remark~3.1]{tau_li}) but we have not found any benefits in
forcing the solution to stay in $\Intdom_{+}^{D}$. However, for
fractional states, the propensities should explicitly be set to zero
whenever executing the reaction would result in a negative copy
number.

\subsection{Stochastic toggle switch}

A \emph{toggle switch} found within the regulatory network of
\textit{E.~coli} has been modeled by two mutually cooperatively
repressing gene products $X$ and $Y$~\cite{toggle_switch}. The model
is
\begin{align}
  \label{eq:rfun4}
  \left. \begin{array}{rclcrcl}
      \emptyset & \xrightarrow{a/(b+y^{2})} & X	& \quad &
      \emptyset & \xrightarrow{a/(b+x^{2})} & Y	                        \\
      X & \xrightarrow{\mu x} & \emptyset & \quad &
      Y & \xrightarrow{\mu y} & \emptyset
    \end{array} \right\},
\end{align}
with parameters $a = 3000$, $b = 11000$ and $\mu =
10^{-3}$. In~\eqref{eq:rfun4}, note that if the number of
$X$-molecules is larger than the number of $Y$-molecules, then the
production of $Y$-molecules is inhibited and the system finds a stable
state with $x > y$. However, the intrinsic noise can make the roles of
$X$ and $Y$ suddenly switch to a state where instead the production of
$X$-molecules is inhibited. The system~\eqref{eq:rfun4} thus
constitutes an interesting example where the deterministic dynamics
clearly differs from that obtained by a stochastic simulation.

In Figure~\ref{fig:rfun4_sol} the exact trajectory up to final time $T
= 5 \times 10^{6}$ is displayed together with the approximation
obtained after a few iterations by the parareal algorithm on a
\emph{very} coarse grid with $\dt = T/50 = 10^{5}$. Interestingly, the
next correct place to switch is found for each new iteration so that
all such events are correctly located within the first 4
iterations. Although the numerical solution sometimes overshoots,
\emph{some} information is evidently being correctly propagated
through the system.

\begin{figure}[htp]
  \centering
  \includegraphics[width = 14cm]{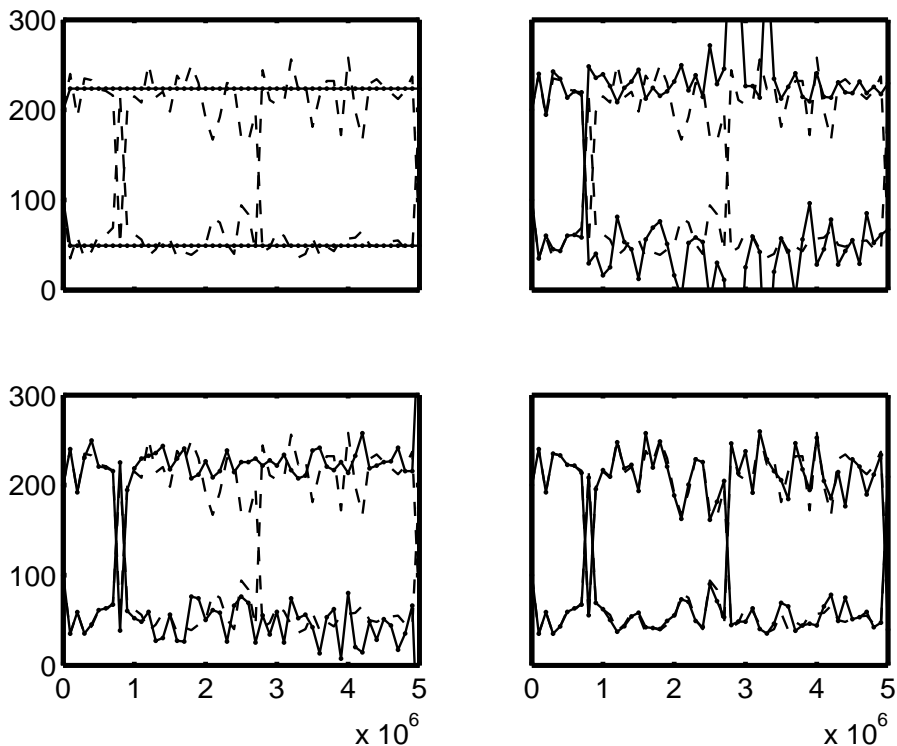}
  \caption{Dashed: exact trajectory of the two components of the
  toggle switch~\eqref{eq:rfun4} versus time. Solid: solution obtained
  in parallel after 0, 1, 2 and 4 iterations. Top left: the initial
  solution from the reaction rate equations immediately settles at one
  of the stable states. Top right and bottom left: the first and
  second switches are correctly obtained after respectively one and
  two iterations. Bottom right: all 4 spontaneous changes of state
  have been correctly detected. Note the rather large level of noise.}
  \label{fig:rfun4_sol}
\end{figure}

The relative errors and residuals are plotted in
Figure~\ref{fig:rfun4_err}. It is seen that the convergence initially
is very rapid but then reaches a plateau where the error stays within
a small fraction of the order of the noise level. The accuracy thus
obtained is quite reasonable and is at the order of a 1\% perturbation
of the propensities in~\eqref{eq:rfun4}.

\begin{figure}[htp]
  \centering
  \includegraphics[width = 10cm]{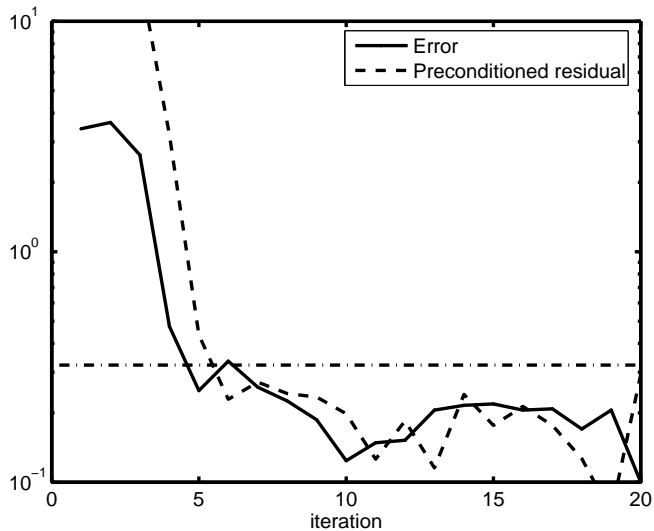}
  \caption{Maximum relative errors and residuals
  (see~\eqref{eq:presidual} and \eqref{eq:error}) obtained for the
  toggle switch during the first 20 iterations. The accuracy is
  quickly improved during the first 5 iterations and then settles more
  slowly. The horizontal line (dash-dot) is the induced difference
  when the propensities are perturbed by $\pm1\%$.}
  \label{fig:rfun4_err}
\end{figure}

\subsection{Homogenization of disparate rates}
\label{subsec:stiff}

As a specific model containing two different scales we consider
\emph{fast dimerization} combined with \emph{slow isomerization},
\begin{align}
  \label{eq:rfun8_5}
  \left. \begin{array}{rcl}
    X_{1}+X_{1} & \overset{1/\varepsilon}{{\rightleftharpoons}} & 
      X_{2}+X_{2}                                                       \\
    X_{2} & \overset{1}{{\rightleftharpoons}} & Y_{2}                   \\
    Y_{2}+Y_{2} & \overset{1/\varepsilon}{{\rightleftharpoons}} & 
      Y_{1}+Y_{1}
  \end{array} \right\},
\end{align}
where the small parameter $\varepsilon$ controls the difference in
scales. For this example we took $\varepsilon = 10^{-3}$ and initial
data $[x_{1},x_{2},y_{1},y_{2}](0) = [15,5,30,10]$ with final time $T
= 10$. The fast channels thus fire about $10^{4}$ times more often
than the slow ones and we note also that the fast channels have
quadratic propensities (e.g.~$w_{1} = x_{1}(x_{1}-1)/\varepsilon$) so
that the rate equations are inexact.

For the parareal discretization we used $\dt = T/50$ as before and an
extremely simple coarse solver in the form of a single step with the
linearized backward Euler method. The homogenization procedure
described in Section~\ref{subsec:homogenization} was used with the
fine solver defined in~\eqref{eq:Fineh} using $\delta t = \dt/2$. The
resulting combination was very effective indeed in obtaining a
homogenized solution, see Figure~\ref{fig:rfun8_5_sol}
and~\ref{fig:rfun8_5_err}.

\begin{figure}[htp]
  \centering
  \includegraphics[width = 14cm]{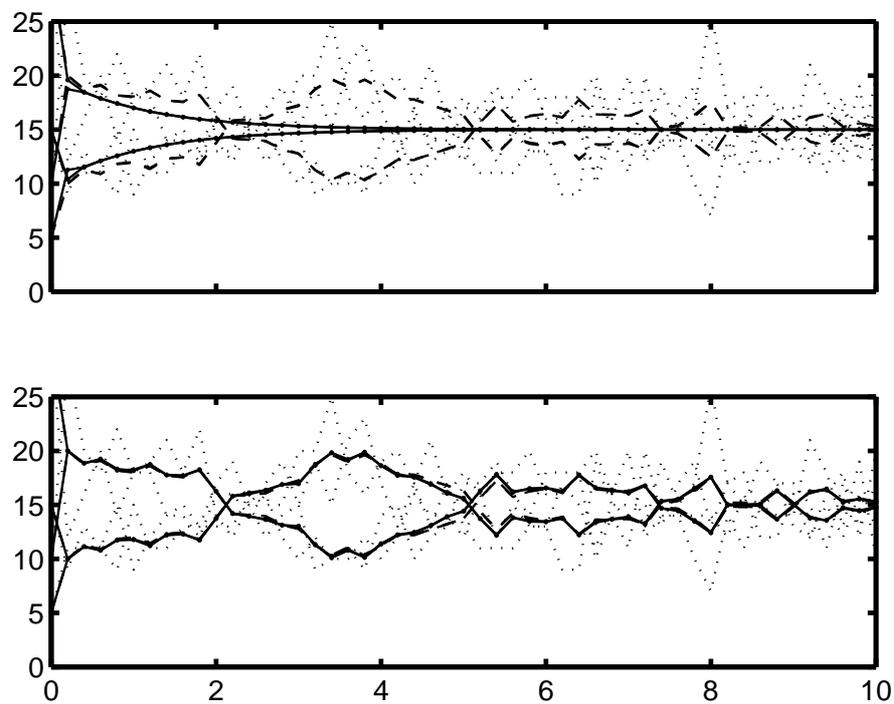}
  \caption{Dashed: homogenized solution for~\eqref{eq:rfun8_5} versus
  time. Solid: solution after 0 (top) and 1 (bottom) iteration of the
  parareal algorithm. For comparison, the corresponding
  non-homogenized trajectory is also plotted using a dotted line. For
  this example, a single step of the backward Euler method was a
  sufficient accurate coarse scale solver which therefore is extremely
  cheap to evaluate.}
  \label{fig:rfun8_5_sol}
\end{figure}

\begin{figure}[htp]
  \centering
  \includegraphics[width = 10cm]{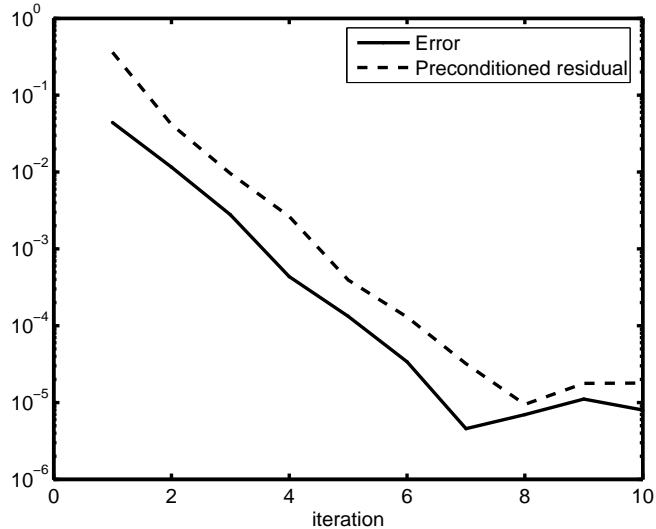}
  \caption{Convergence history for~\eqref{eq:rfun8_5} homogenized
  according to~\eqref{eq:Fineh}.}
  \label{fig:rfun8_5_err}
\end{figure}

\subsection{Stochastic reaction-diffusion}
\label{subsec:rdme}

There are many interesting chemical systems where the spatial
distribution of the species must be taken into account. The usual
thermodynamical equilibria assumption (``well stirredness'') no longer
holds as the transport of the molecules through the solvent is slow
compared to average reaction times or when some reactions are strongly
localized. The length scale over which the system can be regarded as
homogeneous is now much shorter.

Examples of when both the stochasticity of the reactions and the
spatial distribution are necessary to explain experimental data are
found in \cite{Dobrzynski, bistable_RDME, NSM, KruseSpat,
  noisyFuture}. For such systems, the diffusion at a molecular level
can be treated as a special set of linear propensities. This yields
the reaction-diffusion master equation (RDME) \cite[Ch.~8]{Gardiner},
\cite[Ch.~XIV]{VanKampen} which evolves the probability density of the
system in the same manner as the CME~\eqref{eq:Master}. However, the
dimensionality of the state-space is much higher and computing even a
single trajectory can be a very computationally intensive
problem. Note also that, as the rate equations now form the
reaction-diffusion PDE, the master operator defines a continuous
spectrum of scales so that a clear separation into slow/fast ones is
no longer possible.

We shall consider a small example in the present section as
follows. For $i = 1 \ldots 5$, denote by the triplet $(x,y,z)_{i}$ the
number of $X$-, $Y$- and $Z$-molecules in cell $i$. The cells each
have volume $\Omega$ and for simplicity are connected in an array
1-2-3-4-5. Within each cell $i$ we specify the \emph{reactions}
\begin{align}
  \label{eq:reactions}
  \left. \begin{array}{lll}
    X_{i}+Y_{i} &\xrightarrow{k_{a} x_{i}y_{i}/\Omega}& Z_{i} \\
    Z_{i} &\xrightarrow{k_{d} z_{i}}& X_{i}+Y_{i}
  \end{array} \right\},
\end{align}
compare~\eqref{eq:el1}--\eqref{eq:el4}. Furthermore, the molecules can
\emph{diffuse} to any neighboring cell $j = i \pm 1 \in
\{1,\ldots,5\}$ according to
\begin{align}
  \label{eq:diffusion}
  \left. \begin{array}{lll}
    X_{i} &\xrightarrow{d x_{i}/h^{2}}& X_{j} \\
    Y_{i} &\xrightarrow{d y_{i}/h^{2}}& Y_{j} \\
    Z_{i} &\xrightarrow{d z_{i}/h^{2}}& Z_{j}
  \end{array} \right\},
\end{align}
where $h \equiv \Omega^{1/3}$ is the length-scale. Finally, there are
also in- and outflow at the boundaries,
\begin{align}
  \label{eq:bflow}
  \left. \begin{array}{lllll}
    \emptyset &\xrightarrow{d N_{\Omega}/h^{2}}& X_{1} 
    \qquad & Z_{1} &\xrightarrow{d z_{1}/h^{2}} \emptyset \\
    \emptyset &\xrightarrow{d N_{\Omega}/h^{2}}& Y_{5} 
    \qquad & Z_{5} &\xrightarrow{d z_{5}/h^{2}} \emptyset
  \end{array} \right\}.
\end{align}

In order to fully describe the system we have chosen constants to
approximately coincide with those in \cite{bistable_RDME} for
\textit{E.~coli}, see also \cite{Berg_review}. The
model is thus parameterized by the integer $N_{\Omega}$ defining the
total volume through $V = 10^{-15} \times N_{\Omega}/25 = 5\Omega$
$(l)$. The rate constants are given by $k_{a} = 10^{8}$
$(M^{-1}s^{-1})$, $k_{d} = 10$ $(s^{-1})$ and $d = 10^{-10}$
$(m^{2}s^{-1})$. Since we have that $N_{\Omega} \propto \Omega$, for
convenience, we loosely refer to both quantities as the ``system
size''. The model is normalized around $N_{\Omega} = 25$ molecules per
cell and as initial data we simply took $(x_{i},y_{i},z_{i}) =
(N_{\Omega},N_{\Omega},0)$ in all cells.

For the parareal algorithm we again took $N = 50$ intervals with total
time $T = 1$. In the semi-discrete case, one can show that the
reaction-rate equations for the diffusion part~\eqref{eq:diffusion}
are equivalent to a mass-lumped FEM-method for the macroscopic
diffusion equation; this observation has been used to generalize the
mesoscopic model to more complicated geometries
\cite{master_spatial}. We note in passing that this setting opens
up for extremely efficient coarse-grain solvers based on multigrid
techniques when the spatial resolution increases.

In Figure~\ref{fig:rfun11_reduct} the error during each parareal step
is displayed. There is a trend with faster convergence as the system
size increases and despite the seemingly rather large error levels,
the solution obtained is visually very pleasing and can hardly be
distinguished from the exact one (cf.~Figure~\ref{fig:rfun11_sol}). In
Figure~\ref{fig:rfun11h_reduct} we have used instead the homogenized
fine scale solver~\eqref{eq:Fineh} and the convergence improves a
lot. Again, smoothing the output from the fine solver by integrating
over a short period of time has the effect of scaling down the
effective Lipschitz constant and convergence to the homogenized
solution becomes fast.

\begin{figure}[htp]
  \centering
  \includegraphics[width = 12cm]{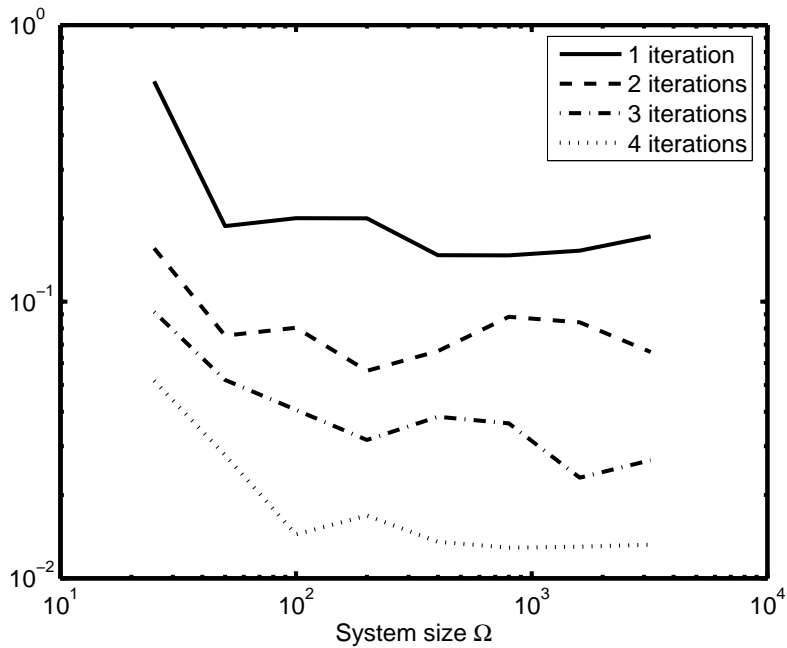}
  \caption{Relative error per the first few parareal iterations when
  the system size $\Omega$ increases.}
  \label{fig:rfun11_reduct}
\end{figure}

\begin{figure}[htp]
  \centering
  \includegraphics[width = 12cm]{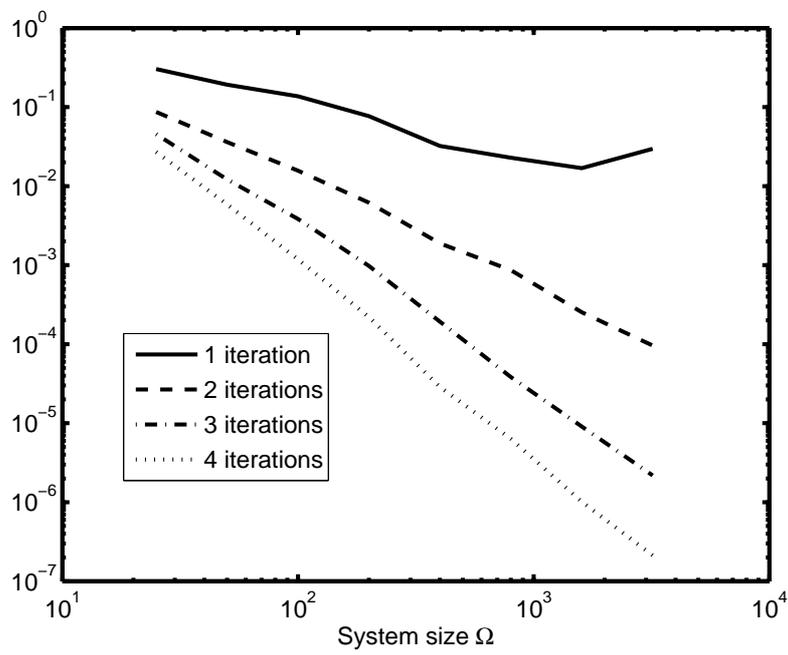}
  \caption{As in Figure~\ref{fig:rfun11_reduct} but using a
  homogenized fine solver~\eqref{eq:Fineh} with $\delta t =
  \dt/4$. The convergence to the homogenized solution is faster and
  more regular.}
  \label{fig:rfun11h_reduct}
\end{figure}

\begin{figure}[htp]
  \centering
  \includegraphics[width = 12cm]{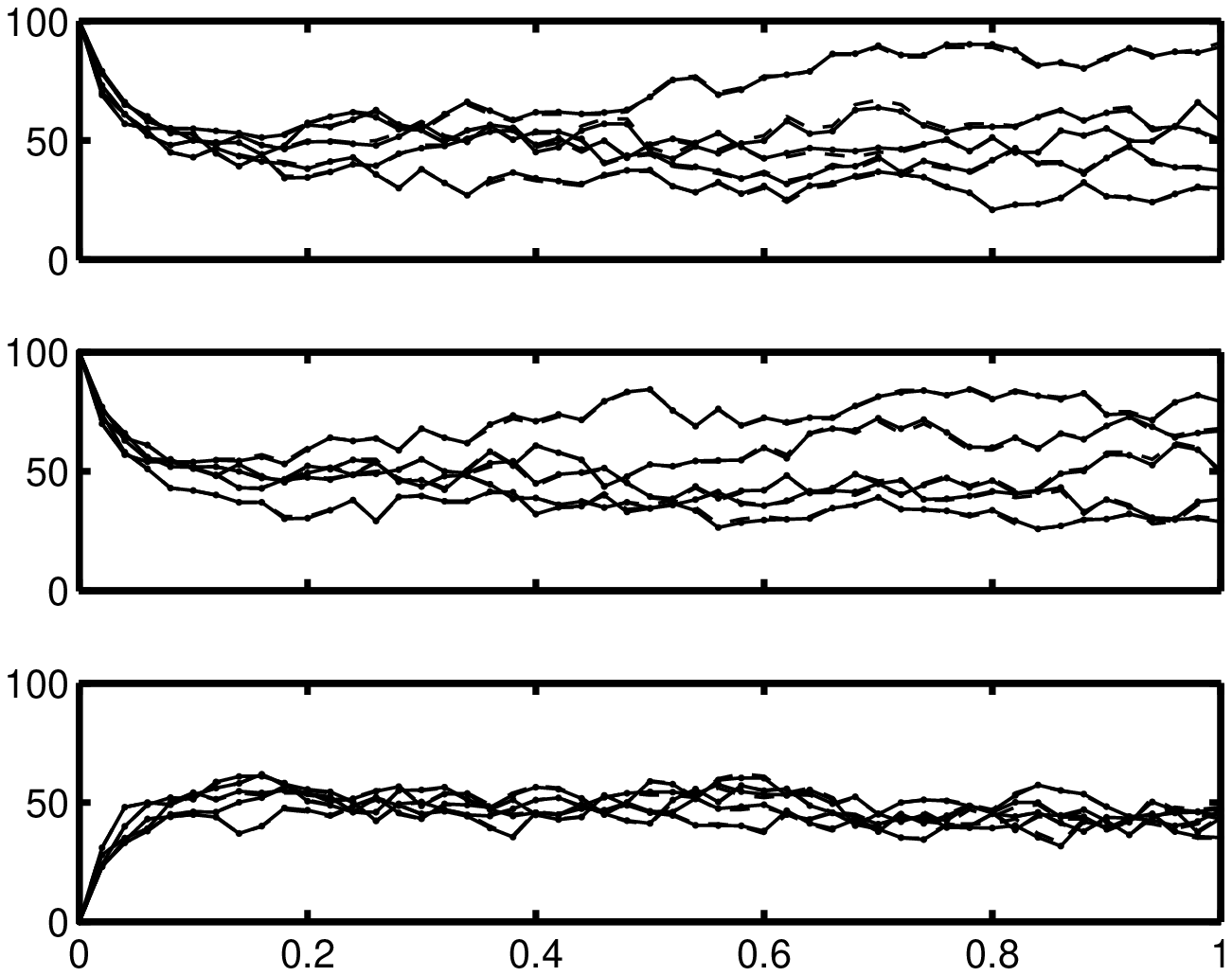}
  \caption{Sample solution for the reaction-diffusion
  system~\eqref{eq:reactions},~\eqref{eq:diffusion}
  and~\eqref{eq:bflow}. The system-size parameter is for this example
  $N_{\Omega} = 100$ molecules per cell. Solid: solution obtained
  after 4 parareal iterations over 50 processors, dashed and barely
  visible: exact trajectory. From top to bottom are the trajectories for
  the $X$-, the $Y$- and the $Z$-molecules, respectively. Towards the
  end of the simulation there is a strong negative correlation between
  $X$- and $Y$-molecules and there is also a spatial correlation due
  to the unsymmetric inflow~\eqref{eq:bflow}.}
  \label{fig:rfun11_sol}
\end{figure}


\section{Conclusions}
\label{sec:concl}

The master equation, or equivalently, a special type of jump SDEs
accurately describes well stirred chemical reactions at the mesoscopic
level. Diffusion can be incorporated by a special set of reactions
implying that non-homogenous problems also can be modeled using the
same type of process. In the macroscopic limit the usual rate
equations (or the reaction-diffusion PDE) emerges.

The parareal algorithm can be applied by using the rate equations as a
predictor for the jump process. Convergence is then dictated by the
Lipschitz constant of the system, but also by the level of noise in
the solution. This noise vanishes in the macroscopic limit.

For models involving rapid transients the best one can generally hope
for is convergence in moment, but it is then not clear how the error
should be monitored. A remedy is to homogenize the model on the fly by
averaging the process in time. This homogenization is very general and
should be applicable to a wide range of stochastic time-dependent
problems. We remark that obtaining this homogenized solution is
impractical on a serial computer since the full solution must be
obtained first. Furthermore, the proposed homogenization can be put in
contrast to other multiscale methods \cite{smalltimesteps,
nestedSSA}. Typically one computes a solution to a homogenized
\emph{equation} obtained by determining some kind of averaged
coefficients. With the parareal algorithm, one rather directly obtains
a homogenized \emph{solution}.

The range of applicability of the method can be increased by adding a
white-noise term to the governing equations, thus encompassing highly
general L\'{e}vy-type SDEs \cite{LevySDEs}. In this general setting,
the suggested homogenization provides a way of applying the parareal
algorithm to stiff problems.

We would also like to relate our results to those in reference
\cite{pararealSDE} mentioned in the introduction. There it is shown
that the parareal/Euler forward combination for a \emph{nonstiff}
Wiener SDE yields a mean square estimate $\Ordo{\dt^{k/2}}$ for the
$k$th parareal iteration assuming a practically exact fine
solver. This holds true in the present context as well provided that
we use the appropriate version of the forward Euler method, namely the
tau-leap method \cite{tau_li}. \emph{However}, the resulting method is
highly nontrivial, if possible at all, to implement efficiently. The
reason is that the same realizations of Poisson processes have to be
used for both the coarse and the fine solver. The list of events
simulated by the fine solver must somehow be reused and searched
through leading to a very expensive coarse solver. An open question is
thus if one can somehow circumvent this issue.

One of the most promising applications of the proposed method seems to
be for reaction-diffusion models such as the one investigated in
Section~\ref{subsec:rdme}. Here the macroscopic model can typically be
acceptable in many, but not all, subvolumes. Some of the species are
typically present in fairly large copy numbers where noise is less
pronounced. A few iterations of the suggested parareal algorithm can
thus be understood as a kind of deterministic/stochastic hybrid
method. A feature with this set-up is that one never needs to
explicitly determine what parts should be treated deterministically.


\section*{Acknowledgment}

Part of this work was initiated during a 7 month visit to the Division
of Applied Mathematics at Brown University, Providence, Rhode
Island. During this visit Jan Hesthaven and his family kindly provided
with both practical and personal support. At Brown University, Yvon
Maday supplied several references on the parareal algorithm.

During the time of writing, Ingemar Kaj offered valuable initial
suggestions and a much appreciated discussion. Finally, Per Lötstedt
and Andreas Hellander made several comments that improved the quality
of the manuscript.


\bibliographystyle{plain}

\begin{thebibliography}{10}

\bibitem{postleap}
D.~F. Anderson.
\newblock Incorporating postleap checks in tau-leaping.
\newblock {\em J.~Chem.~Phys}, 128(5):054103--054111, 2008.
\newblock \texttt{doi:10.1063/1.2819665}.

\bibitem{LevySDEs}
D.~Applebaum.
\newblock {\em L\'{e}vy Processes and Stochastic Calculus}, volume~93 of {\em
  Cambridge Studies in Advanced Mathematics}.
\newblock Cambridge University Press, Cambridge, 2004.

\bibitem{pararealMol}
L.~Baffico, S.~Bernard, Y.~Maday, G.~Turinici, and G.~Z\'{e}rah.
\newblock Parallel-in-time molecular-dynamics simulations.
\newblock {\em Phys.~Rev.~E}, 66(5):057701--057704, 2002.
\newblock \texttt{doi:10.1103/PhysRevE.66.057701}.

\bibitem{stabPararealPDE}
G.~Bal.
\newblock On the convergence and the stability of the parareal algorithm to
  solve partial differential equations.
\newblock In {\em Domain Decomposition Methods in Science and Engineering},
  volume~40 of {\em Lecture Notes in Computational Science and Engineering},
  pages 425--432. Springer, Berlin, 2005.
\newblock \texttt{doi:10.1007/3-540-26825-1\_43}.

\bibitem{pararealSDE}
G.~Bal.
\newblock Parallelization in time of (stochastic) ordinary differential
  equations, 2006.
\newblock Submitted. Available at
  \texttt{http://www.columbia.edu/$\sim$gb2030}.

\bibitem{Berg_review}
O.~G. Berg and P.~H. von Hippel.
\newblock Diffusion-controlled macromolecular interactions.
\newblock {\em Ann.~Rev.~Biophys.~Chem.}, 14(1):131--160, 1985.
\newblock \texttt{doi:10.1146/annurev.bb.14.060185.001023}.

\bibitem{smalltimesteps}
Y.~Cao, D.~Gillespie, and L.~Petzold.
\newblock Multiscale stochastic simulation algorithm with stochastic partial
  equilibrium assumption for chemically reacting systems.
\newblock {\em J.~Comput.~Phys.}, 206:395--411, 2005.
\newblock \texttt{doi:10.1016/j.jcp.2004.12.014}.

\bibitem{stab_tau_leap}
Y.~Cao, L.~R. Petzold, M.~Rathinam, and D.~T. Gillespie.
\newblock The numerical stability of leaping methods for stochastic simulation
  of chemically reacting systems.
\newblock {\em J.~Chem.~Phys.}, 121(24):12169--12178, 2004.
\newblock \texttt{doi:10.1063/1.1823412}.

\bibitem{pointProc}
D.~R. Cox and V.~Isham.
\newblock {\em Point processes}.
\newblock Monographs on applied probability and statistics. Chapman and Hall,
  London, 1980.

\bibitem{Dobrzynski}
M.~Dobrzy{\'n}ski, J.~V. Rodr{\'\i}guez, J.~A. Kaandorp, and J.~G. Blom.
\newblock Computational methods for diffusion-influenced biochemical reactions.
\newblock {\em Bioinformatics}, 23(15):1969--1977, 2007.
\newblock \texttt{doi:10.1093/bioinformatics/btm278}.

\bibitem{nestedSSA}
W.~E, D.~Liu, and E.~Vanden-Eijnden.
\newblock Nested stochastic simulation algorithm for chemical kinetic systems
  with disparate rates.
\newblock {\em J.~Chem.~Phys.}, 123(19), 2005.
\newblock \texttt{doi:10.1063/1.2109987}.

\bibitem{bistable_RDME}
J.~Elf and M.~Ehrenberg.
\newblock Spontaneous separation of bi-stable biochemical systems into spatial
  domains of opposite phases.
\newblock {\em Syst.~Biol.}, 1(2):230--236, 2004.
\newblock \texttt{doi:10.1049/sb:20045021}.

\bibitem{stochGene}
M.~B. Elowitz, A.~J. Levine, E.~D. Siggia, and P.~S. Swain.
\newblock Stochastic gene expression in a single cell.
\newblock {\em Science}, 297(5584):1183--1186, 2002.
\newblock \texttt{doi:10.1126/science.1070919}.

\bibitem{master_spatial}
S.~Engblom, L.~Ferm, A.~Hellander, and P.~L\"{o}tstedt.
\newblock Simulation of stochastic reaction-diffusion processes on unstructured
  meshes.
\newblock Technical Report 2008-012, Dept of Information Technology, Uppsala
  University, Uppsala, Sweden, 2008.
\newblock Available at \texttt{http://www.it.uu.se/research}.

\bibitem{Markovappr}
S.~N. Ethier and T.~G. Kurtz.
\newblock {\em Markov processes: characterization and convergence}.
\newblock Wiley series in Probability and Mathematical Statistics. John Wiley
  \& Sons, New York, 1986.

\bibitem{NSM}
D.~Fange and J.~Elf.
\newblock Noise-induced min phenotypes in \textit{E.~coli}.
\newblock {\em PLoS Comput.~Biol.}, 2(6):637--648, 2006.
\newblock \texttt{doi:10.1371/journal.pcbi.0020080}.

\bibitem{pararealSurv}
M.~J. Gander and S.~Vandewalle.
\newblock Analysis of the parareal time-parallel time-integration method.
\newblock {\em SIAM J.~Sci.~Comput.}, 29(2):556--578, 2007.
\newblock \texttt{doi:10.1137/05064607X}.

\bibitem{Gardiner}
C.~W. Gardiner.
\newblock {\em Handbook of Stochastic Methods}.
\newblock Springer Series in Synergetics. Springer-Verlag, Berlin, 3rd edition,
  2004.

\bibitem{toggle_switch}
T.~S. Gardner, C.~R. Cantor, and J.~J. Collins.
\newblock Construction of a genetic toggle switch in {E}scherichia coli.
\newblock {\em Nature}, 403:339--342, 2000.
\newblock \texttt{doi:10.1038/35002131}.

\bibitem{gillespiemodern}
M.~A. Gibson and J.~Bruck.
\newblock Efficient exact stochastic simulation of chemical systems with many
  species and many channels.
\newblock {\em J.~Phys.~Chem.}, 104(9):1876--1889, 2000.
\newblock \texttt{doi:10.1021/jp993732q}.

\bibitem{SDEs}
I.~I. Gihman and A.~V. Skorohod.
\newblock {\em Stochastic differential equations}, volume~72 of {\em Ergebnisse
  der Mathematik und ihrer Grenzgebiete}.
\newblock Springer, Berlin, 1972.

\bibitem{gillespie}
D.~T. Gillespie.
\newblock A general method for numerically simulating the stochastic time
  evolution of coupled chemical reactions.
\newblock {\em J.~Comput.~Phys.}, 22(4):403--434, 1976.
\newblock \texttt{doi:10.1016/0021-9991(76)90041-3}.

\bibitem{gillespieCME}
D.~T. Gillespie.
\newblock A rigorous derivation of the chemical master equation.
\newblock {\em Physica A}, 188:404--425, 1992.
\newblock \texttt{doi:10.1016/0378-4371(92)90283-V}.

\bibitem{tau_leap}
D.~T. Gillespie.
\newblock Approximate accelerated stochastic simulation of chemically reacting
  systems.
\newblock {\em J.~Chem.~Phys.}, 115(4):1716--1733, 2001.
\newblock \texttt{doi:10.1063/1.1378322}.

\bibitem{guptasarama}
P.~Guptasarama.
\newblock Does replication-induced transcription regulate synthesis of the
  myriad low copy number proteins of \textit{{E}scherichia coli}?
\newblock {\em Bioessays}, 17(11):987--997, 1995.
\newblock \texttt{doi:10.1002/bies.950171112}.

\bibitem{haseltine_HSSA}
E.~L. Haseltine and J.~B. Rawlings.
\newblock Approximate simulation of coupled fast and slow reactions for
  stochastic chemical kinetics.
\newblock {\em J.~Chem.~Phys.}, 117(15), 2002.
\newblock \texttt{doi:10.1063/1.1505860}.

\bibitem{SDEsDiffusion}
N.~Ikeda and S.~Watanabe.
\newblock {\em Stochastic Differential Equations and Diffusion Processes},
  volume~24 of {\em North-Hollans Mathematical Library}.
\newblock North-Holland/Kodansha, Amsterdam/Tokyo, 1981.

\bibitem{stochLimits}
J.~Jacod and A.~N. Shiryaev.
\newblock {\em Limit Theorems for Stochastic Processes}, volume 288 of {\em
  Grundlehren der mathematischen Wissenschaften}.
\newblock Springer, Berlin, 1987.

\bibitem{VanKampen}
{\noopsort{Kampen}}{N.~G.~van Kampen}.
\newblock {\em Stochastic Processes in Physics and Chemistry}.
\newblock Elsevier, Amsterdam, 5th edition, 2004.

\bibitem{KruseSpat}
K.~Kruse and J.~Elf.
\newblock Kinetics in spatially extended systems.
\newblock In Z.~Szallasi, J.~Stelling, and V.~Periwal, editors, {\em System
  Modeling in Cellular Biology. From Concepts to Nuts and Bolts}, pages
  177--198. MIT Press, Cambridge, MA, 2006.

\bibitem{tau_li}
T.~Li.
\newblock Analysis of explicit tau-leaping schemes for simulating chemically
  reacting systems.
\newblock {\em Multiscale Model.~Simul.}, 6(2):417--436, 2007.
\newblock \texttt{doi:10.1137/06066792X}.

\bibitem{stiffSDEimpl}
T.~Li, A.~Abdulle, and W.~E.
\newblock Effectiveness of implicit methods for stiff stochastic differential
  equations.
\newblock {\em Commun.~Comput.~Phys.}, 3(2):295--307, 2008.

\bibitem{parareal1}
J.-L. Lions, Y.~Maday, and G.~Turinici.
\newblock R\'{e}solution d'{EDP} par un sch\'{e}ma en temps ``parar\'{e}el''.
\newblock {\em C.~R.~Acad.~Sci.~Paris~S\'{e}r.~I~Math.}, 332(7):661--668, 2000.
\newblock \texttt{doi:10.1016/S0764-4442(00)01793-6}.

\bibitem{pararealMreduct}
Y.~Maday.
\newblock Parareal in time algorithm for kinetic systems based on model
  reduction.
\newblock In A.~Bandrauk, M.~C. Delfour, and C.~Le Bris, editors, {\em
  High-Dimensional Partial Differential Equations in Science and Engineering},
  volume~41 of {\em CRM Proceedings \& Lecture Notes}, pages 183--194. AMS/CRM,
  2007.

\bibitem{pararealctrl}
Y.~Maday and G.~Turinici.
\newblock A parareal in time procedure for the control of partial differential
  equations.
\newblock {\em C.~R.~Acad.~Sci.~Paris~S\'{e}r.~I~Math.}, 335(4):387--392, 2002.
\newblock \texttt{doi:10.1016/S1631-073X(02)02467-6}.

\bibitem{noisyBusiness}
H.~H. McAdams and A.~Arkin.
\newblock It's a noisy business. genetic regulation at the nanomolar scale.
\newblock {\em Trends Gen.}, 15(2):65--69, 1999.
\newblock \texttt{doi:10.1016/S0168-9525(98)01659-X}.

\bibitem{noisyFuture}
R.~Metzler.
\newblock The future is noisy: The role of spatial fluctuations in genetic
  switching.
\newblock {\em Phys.~Rev.~Lett.}, 87(6):068103--068107, 2001.
\newblock \texttt{doi:10.1103/PhysRevLett.87.068103}.

\bibitem{noisygeneregul}
J.~Paulsson, O.~G. Berg, and M.~Ehrenberg.
\newblock Stochastic focusing: Fluctuation-enhanced sensitivity of
  intracellular regulation.
\newblock {\em Proc.~Natl.~Acad.~Sci.~USA}, 97(13):7148--7153, 2000.
\newblock \texttt{doi:10.1073/pnas.110057697}.

\bibitem{ME2SDE}
S.~Plyasunov.
\newblock On hybrid simulation schemes for stochastic reaction dynamics, 2005.
\newblock Available at \texttt{http://arxiv.org/abs/math/0504477}.

\bibitem{noisygene}
J.~M. Raser and E.~K. O'Shea.
\newblock Noise in gene expression: Origins, consequences and control.
\newblock {\em Science}, 309(5743):2010--2013, 2005.
\newblock \texttt{doi:10.1126/science.1105891}.

\bibitem{impl_tau_leap}
M.~Rathinam, L.~Petzold, Y.~Cao, and D.~T. Gillespie.
\newblock Stiffness in stochastic chemically reacting systems: The implicit
  tau-leaping method.
\newblock {\em J.~Chem.~Phys.}, 119(24):12784--12794, 2003.
\newblock \texttt{doi:10.1063/1.1627296}.

\bibitem{tau_leap_anal}
M.~Rathinam, L.~R. Petzold, Y.~Cao, and D.~T. Gillespie.
\newblock Consistency and stability of tau-leaping schemes for chemical
  reaction systems.
\newblock {\em Multiscale Model.~Simul.}, 4(3):867--895, 2005.
\newblock \texttt{doi:10.1137/040603206}.

\bibitem{parareal}
G.~A. Staff.
\newblock The parareal algorithm, 2003.
\newblock Available at \texttt{www.idi.ntnu.no/$\sim$elster/notur-cluster03}.

\bibitem{pararealAnal}
G.~A. Staff, Y.~Maday, and E.~M. R{\o}nquist.
\newblock The parareal-in-time algorithm: basics, stability analysis and more,
  2006.
\newblock Submitted. Available at
  \texttt{http://www.springerlink.com/index/lp436711n8174312.pdf}.

\bibitem{newsteadystates_RDME}
Y.~Togashi and K.~Kaneko.
\newblock Molecular discreteness in reaction-diffusion systems yields steady
  states not seen in the continuum limit.
\newblock {\em Phys.~Rev.~E}, 70(2):020901--1--020901--4, 2004.
\newblock \texttt{doi:10.1103/PhysRevE.70.020901}.

\bibitem{circadian2D}
J.~M.~G. Vilar, H.~Y. Kueh, N.~Barkai, and S.~Leibler.
\newblock Mechanism of noise-resistance in genetic oscillators.
\newblock {\em Proc.~Nat.~Acad.~Sci.}, 99:5988--5992, 2002.
\newblock \texttt{doi:10.1073/pnas.092133899}.

\end{thebibliography}

\providecommand{\noopsort}[1]{}

\end{document}